\documentclass[a4paper,final]{amsart}
\usepackage{amsfonts}
\usepackage{amssymb}
\usepackage[latin2]{inputenc}
\usepackage{enumerate}
\usepackage{amsmath}
\usepackage[notcite,notref]{showkeys}
\newtheorem{theorem}{Theorem}[section]
\newtheorem*{thmmain}{Theorem \ref{paradoxical}}
\newtheorem{lemma}[theorem]{Lemma}
\newtheorem{claim}[theorem]{Claim}
\newtheorem{prop}[theorem]{Proposition}

\newtheorem{cor}[theorem]{Corollary}

\theoremstyle{definition}                  
\newtheorem{remark}[theorem]{Remark}

\newtheorem*{notation}{Notation and definitions}

\newcommand{\Z}{\mathbb{Z}}

\newcommand{\R}{\mathbb{R}}

\def\phi{\varphi} 

\newcommand{\eps}{\varepsilon}

\newcommand{\iH}{\mathcal{H}}
\newcommand{\iI}{\mathcal{I}}

\newcommand{\semmi}[1]{}
\newcommand{\diam}{\text{diam}}
\newcommand{\norm}[1]{{\left\lVert #1 \right\rVert}}
\newcommand{\abs}[1]{{\left\lvert #1 \right\rvert}}

\author{Andr\'as M\'ath\'e}
\thanks{Supported by a Leverhulme Trust Early Career Fellowship and the Hungarian Scientific Research Fund grant~104178.}

\begin{document}
\title[Measuring sets with translation invariant Borel measures]{Measuring sets with  \\ translation invariant Borel measures}

\maketitle 

\begin{abstract}
Following Davies, Elekes and Keleti, we study \emph{measured} sets, i.e. Borel sets $B$ in $\R$ (or in a Polish group) for which there is a translation invariant Borel measure assigning positive and $\sigma$-finite measure to $B$. We investigate which sets can be written as a (disjoint) union of measured sets.

We show that every Borel nullset $B\subset \R$ of the second category is larger than any nullset $A\subset \R$ in the sense that there are partitions $B=B_1\cup B_2$, $A=A_1\cup A_2$ and gauge functions $g_1, g_2$ such that the Hausdorff measures satisfy $\iH^{g_i}(B_i)=1$ and $\iH^{g_i}(A_i)=0$ ($i=1,2$). This implies that every Borel set of the second category is a union of two measured sets.

We also present Borel and compact sets in $\R$ which are not a union of countably many measured sets. This is done in two steps.
First we show that \mbox{non-locally} compact Polish groups are not a union of countably many measured sets. Then, to certain Banach spaces
we associate a Borel and/or $\sigma$-compact additive subgroup of $\R$ which is not a union of countably many measured sets.

It is also shown that there are measured sets which are null or non-$\sigma$-finite for every Hausdorff measure of arbitrary gauge function.

\end{abstract}

\section{Introduction and main results}

We say that a Borel set $B$ in $\R$ is \emph{measured} if there is a translation invariant Borel measure which assigns positive and $\sigma$-finite measure to $B$.

It is not trivial to exhibit sets which are not measured.
R.~O.~Davies \cite{davies} constructed a non-empty compact set in $\R$ which is of zero or non-$\sigma$-finite measure for every translation invariant Borel measure (i.e.~not measured). Previously, D.~G.~Larman \cite{larman2} gave an example of a $G_\delta$ set which is not measured.

Since Hausdorff measures (of arbitrary gauge functions) are translation invariant, a set which is not measured is necessarily null or non-$\sigma$-finite for every Hausdorff measure.

D.~Mauldin raised the question whether the set of Liouville numbers $L$ is measured or not. This set $L$ is $G_\delta$ and periodic to every rational number. M.~Elekes and T.~Keleti showed \cite{marcitamas} that every such set, and thus $L$, is not measured. They also showed that non-$F_\sigma$ Borel subgroups of $\R$ are not measured either.

In this paper the decomposition question is investigated: which sets can be written as a (disjoint) union of measured sets. M.~Elekes and T.~Keleti raised several questions related to this: whether the union of two measured sets is necessarily measured, or, on the contrary, every Borel set is the union of finitely or countably many measured sets. We answer these in the negative.

In Section~\ref{s1} we present a theorem which shows that Borel sets of the second category are actually large in terms of ``size'' as well, where size means Hausdorff measure. 
\begin{thmmain}
Let $A, B\subset \R$ be Borel sets of zero Lebesgue measure and assume that $B$ is of the second category. Then there are Borel partitions $B=B_1\cup B_2$, $A=A_1\cup A_2$ and gauge functions $g_1, g_2$ such that the Hausdorff measures satisfy 
\begin{align*}
\iH^{g_1}(B_1)=1, &\quad \iH^{g_1}(A_1)=0, \\
\iH^{g_2}(B_2)=1, &\quad \iH^{g_2}(A_2)=0.
\end{align*}
\end{thmmain}
This theorem might seem paradoxical when applied to $A=B$, or when the Hausdorff dimension of $A$ is larger than that of $B$.
The theorem implies that every Borel set of the second category is a union of two measured sets, in fact, two sets which are measured by Hausdorff measures.

In Section~\ref{s3} we exhibit Borel and also compact sets which are not a union of countably many measured sets. This is done through addressing the problem in Banach spaces and Polish groups.
(We say that a Borel set $A$ in a Polish group is measured if there is a (both left and right) translation invariant Borel measure which gives positive and $\sigma$-finite measure to $A$.)
First we show that Polish groups which are not locally compact are not a union of countably many measured sets (Theorem~\ref{thmpolish}). (This can be seen as a variant of the statements that there is no Haar measure on such groups, and that the union of countably many Haar null sets is also Haar null.)
Using this result, to Banach spaces we associate Borel additive subgroups of $\R$ which are not a union of countably many measured sets. For $\ell_p$ spaces with $1\le p<\infty$ (and in general, when the space has a boundedly complete basis) we obtain a $\sigma$-compact additive subgroup of $\R$ which is not a union of countably many measured sets (Theorem~\ref{banachconstr}). 

Before giving this general construction in Section~\ref{s3} using Banach spaces, we give a direct proof in Section~\ref{s2} that there exists a non-empty compact set in $\R$ which is not a union of countably many measured sets (without referring to Banach spaces and Polish groups). The obtained compact set is very similar to those obtained by the general construction for the Banach space $\ell_1$.

Finally, in Section~\ref{s4} we show that the class of measured sets is not the same as the class of sets which are measured by a Hausdorff measure. This also answers a question of M.~Elekes and T.~Keleti. The proof is based on the facts that being measured relies much on the additive structure of a set, while being measured by a Hausdorff measure (of arbitrary gauge function) is bi-Lipschitz invariant. In fact, we give two (types) of examples. An explicit example imitates Davies's construction \cite{davies} but uses algebraically independent numbers (and a theorem of J.~von~Neumann). The other example involves typical $C^1$ images of small perfect sets. 
Results about typical compact sets are also mentioned.

\begin{notation}
By gauge function we mean a monotone increasing right continuous function $g:[0,\infty)\to [0,\infty)$. The Hausdorff measure with gauge function $g$ is defined as
$$\iH^g(A)=\lim_{\delta\to 0+} \ \inf \left\{ \sum_{i=1}^\infty g(\diam\,U_i) \,:\, A\subset \cup_{i=1}^\infty U_i \text{ and } \diam\,U_i <\delta\right\}.$$  

We say that a measure $\mu$ on the Borel subsets of a Polish group is translation invariant if $\mu(gB)=\mu(Bg)=\mu(B)$ for all Borel sets $B$ and group elements $g$.
\end{notation}

\section{Decomposing sets as the union of measured sets}\label{s1}

We start with a powerful observation.

\begin{lemma}\label{felbontas}
Let $B\subset \R$ and let $g_1, g_2$ be two gauge functions such that \begin{equation*}
\iH^{\min(g_1, g_2)}(B) =0.
\end{equation*}
Then there are disjoint sets $B_1$, $B_2$ such that
$$B=B_1\cup B_2 \quad\text{and}\quad\iH^{g_1}(B_1)=\iH^{g_2}(B_2)=0.$$
If $B$ is Borel or analytic then $B_1$ and $B_2$ can be chosen to be Borel or analytic, respectively.
\end{lemma}

\begin{proof}

Let $g(x)=\min(g_1(x), g_2(x))$. Since $\iH^g(B)=0$, we can find countable collections of open intervals $\iI_k$ ($k\ge 1$)
such that $$B\subset \bigcup \iI_k \quad \text{and}\quad \sum_{I\in \iI_k} g(|I|)<2^{-k} \quad (k\ge 1).$$
Based on the length of the intervals we can split each $\iI_k$ as $\iI^1_k \cup \iI^2_k$ where
$$\sum_{I\in \iI^1_k} g_1(|I|) + \sum_{I\in \iI^2_k} g_2(|I|) < 2^{-k}.$$
Let $$B_1=B\cap \bigcap_{n=1}^\infty \bigcup_{k=n}^\infty \bigcup \iI^1_k.$$
Let $B_2=B\setminus B_1$. Then
$$B_2 \,\subset \,\bigcup_{n=1}^\infty \bigcap_{k=n}^\infty \bigcup \iI^2_k
\,\subset\, 
\bigcap_{n=1}^\infty \bigcup_{k=n}^\infty \bigcup \iI^2_k.$$
Clearly, $\iH^{g_1}(B_1)=0$ and $\iH^{g_2}(B_2)=0$.

The rest of the statement follows from the fact that $B_1$ is an intersection of $B$ with a $G_\delta$ set.
\end{proof}

The following statement gives a sufficient condition for a set to be a union of two measured sets.

\begin{lemma}\label{felbontas2}
Let $B\subset \R$ be a Borel (or analytic) set and let $g_1, g_2$ be two gauge functions such that \begin{align*}
\iH^{g_1}(B)& >0, \\
\iH^{g_2}(B)&>0, \\
\iH^{\min(g_1, g_2)}(B)& =0.
\end{align*}
Then $B$ is a union of disjoint Borel (or analytic) sets $B_1, B_2$ with $0<\iH^{g_i}(B_i)<\infty$ ($i=1,2$).
\end{lemma}
\begin{proof}
Let $B'_1$ and $B'_2$ be the sets obtained from Lemma~\ref{felbontas}. Then $B'_{3-i}=B\setminus B'_i$ is analytic, $\iH^{g_i}(B\setminus B'_i)>0$. The inner regularity of (generalised) Hausdorff measures \cite[Theorem 3]{larman} implies that there is a compact set $K_i\subset B\setminus B'_i$ such that $0<\iH^{g_i}(K_i)<\infty$. Let $B_1=(B'_1\cup K_1)\setminus K_2$ and $B_2=(B'_2\cup K_2)\setminus K_1$. Then $B=B_1\cup B_2$ and $0<\iH^{g_i}(B_i)=\iH^{g_i}(K_i)<\infty$.
\end{proof}

The proof of the main theorem of this section will not rely on the cited general inner regularity result.

\begin{prop}\label{g1g2}
Let $B\subset \R$ be a Borel (or analytic) set of the second Baire category and let $A\subset \R$ have Lebesgue measure zero. Then there are gauge functions $g_1, g_2$ such that  
$\iH^{g_1}(B)>0$, $\iH^{g_2}(B)>0$, and $\iH^{\min(g_1, g_2)}(A)=0$.
\end{prop}


The idea of the proof is the following. The set $B$ contains a $G_\delta$ set which is dense in some interval. Inside that, we construct two ``balanced'' compact sets and define a gauge function for each of them such that the corresponding Hausdorff measures are positive (and finite). The constructions should be made such that the resulting gauge functions are incomparable. It does not matter how small the compact sets are if they are small on different scales. Notice that if $g(x)=Cx$ for an arbitrarily large constant $C$, then $\iH^g(A)=0$ since $\iH^g$ is comparable to Lebesgue measure. Let us assume that $g=\min(g_1, g_2)$ is defined on $[y, \infty)$ already.
If we can ensure during the constructions of the compact sets that $$g(x)\approx x\cdot \frac{g(y)}{y}$$
on an interval $[\eps, y]$, where $\eps$ is sufficiently small depending on $y$, $g(y)$ and $A$, then we will be able to achieve $\iH^g(A)=0$.

\begin{proof}
Fix a sequence of open intervals $(I_j)$ such that $\sum \diam\,I_j <\infty$ and every point in $A$ is covered by infinitely many $I_j$. This is possible since $A$ has Lebesgue measure zero.
Let $$c(\delta)=\sum_{\diam\,I_j \le \delta} \diam\,I_j.$$
Then $c(\delta)\to 0$ as $\delta\to 0$.

Every analytic set has the Baire property. Since $B$ is not of the first category, this implies that $B$ contains a $G_\delta$ subset which is dense in some open interval. We may assume that this interval is $(0,1)$. Let $G\subset (0,1)\cap B$ be dense $G_\delta$. Fix a nested decreasing sequence of open sets $G_n\subset (0,1)$ such that $G=\bigcap_n G_n$.

For every positive integer $m$ there exist points $x^m_j\in G_m$ ($j=1, 2, \ldots, m$) such that
$$\left| x^m_j- \frac{j}{m} \right| \le \frac{1}{10m}.$$
Let $r(m)>0$ be such that $$[x^m_j, \,x^m_j+r(m)]\subset G_m.$$
We may assume that $r(m)<1/(10m^2)$.
Let $E_m=\{x^m_1, \ldots x^m_m\}$.
For $i=1,2$, let $(m^i_k)$ be rapidly growing sequences of integers such that
$$m^1_k \ll m^2_k \ll m^1_{k+1} \quad (k\ge 1).$$
Later we will require specific conditions on $(m^i_k)$, but all of them will be satisfied if $m^2_k$ is sufficiently large compared to $m^1_k$, and $m^1_{k+1}$ is sufficiently large compared to $m^2_k$.

For simplicity, let
$$r^i_k=r(m^i_k).$$

Now we fix ``balanced'' compact sets $F^i$ ($i=1,2$) of the form
$$F^i=\bigcap_{k=1}^\infty F^i_k+[0,r^i_k]$$
where $F^i_k\subset E_{m^i_k}$.
Let $F^i_1=E_{m^i_1}$ and let $n^i_1=|F^i_1|$. If $$F^i_k\subset E_{m^i_k}$$ is already defined for some $k\ge 1$, let $F^i_{k+1}$ be a maximal subset of $E_{m^i_{k+1}}$ with the properties that
\begin{itemize}
\item for every $x\in F^i_{k+1}$ there is $y\in F^i_k$ such that
$[x, x+r^i_{k+1}] \subset [y, y+r^i_{k}]$;
\item for all $y\in F^i_k$ the sets $F^i_{k+1}\cap [y, y+r^i_{k}]$ have equal cardinalities. 
\end{itemize}
Let us call this common cardinality $n^i_{k+1}$. Then 
$$|F^i_k| = n^i_1 n^i_2 \cdots n^i_k.$$
Clearly, if $m^i_{k}$ is large compared to $r^i_{k-1}$, then
$${r^i_{k-1}} m^i_k/2 \le n^i_{k} \le 2 r^i_{k-1} m^i_k$$
and
\begin{equation}\label{est}
(r^i_{k-1} m^i_k/2) \lvert F^i_{k-1} \rvert \le |F^i_{k}| \le 2 r^i_{k-1} m^i_k |F^i_{k-1}|.
\end{equation}

Let
$$F^i=\bigcap_{k=1}^\infty F^i_k+[0, r^i_k].$$
This is an intersection of nested compact sets and $F^i\subset G \subset B$.

Define $g_i:[0,r^i_1]\to [0, \infty)$ such that $g_i(0)=0$, 
$$g_i(r^i_k)=\frac{1}{|F^i_k|} \quad (i=1,2 \text{ and } k\ge 1)$$
and $g_i$ is linear on the intervals $[r^i_{k+1}, r^i_k]$.
If the sequences $m^i_k$ tend to infinity fast enough, these gauge functions are strictly increasing; moreover, $r^i_k m^i_k\to 0$ implies that we can also ensure
that
$$\frac{g_i(r^i_k)}{r^i_k}$$ is strictly increasing as $k\to\infty$ for each $i$. (In fact, we can ensure that $g_i$ is concave.)

\begin{claim}\label{mertek}
We have $0<\iH^{g_i}(F^i)\le 1$ ($i=1,2$).
\end{claim}
\begin{proof}[Proof of Claim~\ref{mertek}]
The obvious covering of $F^i$ by $|F^i_k|$ intervals of length $r^i_k$ shows that $\iH^{g_i}(F^i)\le 1$.

Let $\mu^i$ be the unique Borel probability measure on $F^i$ for which 
$$\mu^i([x, x+r^i_k])=\frac{1}{|F^i_k|} \quad (x\in F^i_k).$$
We claim that $$\mu^i(I)\le 8g_i(\diam\,I)$$
for every interval $I$ of length less than $r^i_1$.

If this is true then $\sum_j g_i(\diam\,I_j)\ge 1/8$ for any sequence of intervals $I_j$ of length less than $r^i_1$ covering $F^i$, and thus $\iH^{g_i}(F^i)\ge 1/8$.

If $\diam\,I\ge 1$ there is nothing to prove. Set $r^i_0=1$. Let $k\ge 0$ be such that $r^i_{k+1}\le \diam\,I\le r^i_k$. Either $\diam\,I<1/(2m^i_{k+1})$ or not. If $\diam\,I\le 1/(2m^i_{k+1})$, then $I$ can intersect only one of the intervals $[x,x+r^i_{k+1}]$ with $x\in F^i_{k+1}$. Therefore 
$$\mu^i(I)\le \frac{1}{|F^i_{k+1}|}=g_i(r^i_{k+1}) \le g_i(\diam\,I).$$

Now assume $1/(2m^i_{k+1}) \le \diam\,I \le r^i_k$.
The minimal distance among points of $E_{m^i_{k+1}}$ is at least $0.8/m^i_{k+1}$, so the same applies to $F^i_{k+1}$. Therefore $\diam\,I$ can intersect at most
$$1+2m^i_{k+1} \diam\,I\le 4m^i_{k+1}\diam\,I$$ intervals of the form $[x,x+r^i_{k+1}]$ with $x\in F^i_{k+1}$. This implies, by \eqref{est}, that
$$\mu^i(I) \le 4m^i_{k+1} (\diam\,I) \frac{1}{|F^i_{k+1}|}\le \frac{4m^i_{k+1} \,\diam\,I}{r^i_{k}m^i_{k+1}|F^i_{k}|/2}
\le
\frac{4 \,\diam\,I}{r^i_k |F^i_k|/2}
.$$ 
Notice that $$g_i(\diam\,I)\ge \frac{\diam\,I}{r^i_k} g_i(r^i_k)=\frac{\diam\,I}{r^i_k |F^i_k|}.$$
Therefore
$$\mu^i(I) \le 8 g_i(\diam\,I).$$
This proves that $\iH^{g_i}(F^i)\ge 1/8$.
\end{proof}

\begin{claim}\label{nulla}
We have $\iH^g(A)=0$ where $g=\min(g_1, g_2)$.
\end{claim}
\begin{proof}[Proof of Claim~\ref{nulla}]
Let $$\rho^i_{k+1}=r^i_k \frac{g_i(r^i_{k+1})}{g_i(r^i_k)}.$$
Then $r^i_{k+1}<\rho^i_{k+1} < r^i_k$.
If $m^1_{k+1}$ is chosen large enough compared to $m^2_k$ (and $1/r^2_k$), and $m^2_k$ is chosen large enough compared to $m^1_k$ (and $1/r^1_k$), then we can have
\begin{equation}\label{cover}
\rho^2_{k+1}<\rho^1_{k+1}<r^2_k<r^1_k \text{ and } \rho^1_{k+1}<\rho^2_{k} < r^1_k<r^2_{k-1},
\end{equation}
and 
\begin{equation}\label{m2klarge}
\frac{c(\rho^1_{k+1})}{r^2_k}\le 2^{-k},
\end{equation}
and
\begin{equation}\label{m1klarge}
\frac{c(\rho^2_{k+1})}{r^1_{k+1}}\le 2^{-k}.
\end{equation}

Recall that $$\frac{g_i(r^i_{k+1})}{r^i_{k+1}}>\frac{g_i(r^i_k)}{r^i_k}.$$
This implies that
for $x\in [r^i_{k+1}, r^i_1]$ we have
$$g_i(x)\le g_i(r^i_{k+1}) + x\cdot \frac{g_i(r^i_k)}{r^i_k} = \rho^i_{k+1}\cdot\frac{ g_i(r^i_k)}{r^i_k} + x  \cdot \frac{g_i(r^i_k)}{r^i_k}.$$
Therefore, for every $x\in [\rho^i_{k+1}, r^i_1]$,
\begin{equation}\label{triv}
g_i(x)\le 2x \cdot \frac{g_i(r^i_k)}{r^i_k}.
\end{equation}

Recall the intervals $I_j$ which are covering $A$ infinitely many times. Let
\begin{align*}
\iI^1_k & =\{I : I=I_j \text{ and } \rho^2_k < \diam\,I \le \rho^1_k\},\\
\iI^2_k & =\{I : I=I_j \text{ and } \rho^1_{k+1} < \diam\,I \le \rho^2_k\}.
\end{align*}
We have $c(\rho^i_k)\ge \sum\{\diam\,I: I\in\iI^i_k\}$. Therefore, \eqref{triv} and then \eqref{m2klarge} or \eqref{m1klarge} implies that
\begin{align*}
\sum_{I\in \iI^1_{k}} g(\diam\,I) & \le 2c(\rho^1_k)\frac{g_2(r^2_{k-1})}{r^2_{k-1}}
\le 2 \cdot 2^{-(k-1)} g_2(r^2_{k-1}) \le 2^{-k+2},
\\
\sum_{I\in \iI^2_{k}} g(\diam\,I) & \le 2c(\rho^2_{k})\frac{g_1(r^1_k)}{r^1_k}
\le 2\cdot 2^{-(k-1)}g_1(r^1_k) \le 2^{-k+2}.
\end{align*}
Since $A$ is covered by $$\bigcup_{k=n}^\infty \left(\bigcup \iI^1_k \cup \bigcup \iI^2_{k}\right)$$
for every $n$, and $\sum_{k=n}^\infty 2^{-k+3}\to 0$ as $n\to \infty$, we obtain $\iH^g(A)=0$.
\end{proof}

These two claims conclude the proof of Proposition~\ref{g1g2}.
\end{proof}

We obtain our main theorem as a corollary of Propostition~\ref{g1g2} and Lemma~\ref{felbontas}.

\begin{theorem}\label{paradoxical}
Let $A, B\subset \R$ be Borel (or analytic) sets of zero Lebesgue measure and assume that $B$ is of the second category. Then there are Borel (or analytic) partitions $B=B_1\cup B_2$, $A=A_1\cup A_2$ 
and gauge functions $g_1, g_2$ such that the Hausdorff measures satisfy 
\begin{align*}
\iH^{g_1}(B_1)=1, &\quad \iH^{g_1}(A_1)=0, \\
\iH^{g_2}(B_2)=1, &\quad \iH^{g_2}(A_2)=0.
\end{align*}
\end{theorem}

\begin{proof}
Use Proposition~\ref{g1g2} for the sets $B'=B$ and $A'=A\cup B$. We obtain gauge functions $g_1, g_2$ such that $\iH^{\min(g_1, g_2)}(A\cup B)=0$ and that $\iH^{g_i}(B)>0$.

Applying Lemma~\ref{felbontas} to $A$ gives a Borel (or analytic) partition $A=A_1\cup A_2$ with $\iH^{g_i}(A_i)=0$.

Applying Lemma~\ref{felbontas2} to $B$ gives a Borel (or analytic) partition $B=B_1\cup B_2$ with $0<\iH^{g_i}(B_i)<\infty$. Renormalising $g_i$ we get $\iH^{g_i}(B_i)=1$. 

Notice that in the proof of Proposition~\ref{g1g2} we actually constructed the compact sets which we use in the proof of Lemma~\ref{felbontas2}, therefore the inner regularity property of general Hausdorff measures for analytic sets is not needed.
\end{proof}

\begin{cor}\label{uninoftwo}
Every Borel set $B\subset \R$ of the second category is a union of two disjoint Borel sets which are measured by Hausdorff measures.
\end{cor}

\begin{proof}
If $B$ has positive Lebesgue measure, then the statement is obvious, as Lebesgue measure is also a Hausdorff measure. Otherwise apply Theorem~\ref{paradoxical} with $A=\emptyset$.
\end{proof}

\bigskip
\section{Sets which cannot be written as a union of measured sets}\label{s2}

\begin{lemma}\label{fubini-mix}
Let $J$ be a countable set. Let $\mu_j$ ($j\in J$) be translation invariant Borel measures on $\R$. Let $K_j\subset [0, a_j]$ be compact sets with $\mu_j(K_j)=1$. Assume that $$\sum_j a_j<\infty.$$ Let $$K=\sum_{j\in J} K_j.$$

Let $B$ be a Borel set containing uncountably many disjoint translates of $K$. Then there are no Borel sets $B_j$ with $B=\cup_j B_j$ where  every $B_j$ has $\sigma$-finite $\mu_j$-measure.
\end{lemma}


\begin{remark}
It will be shown later that this lemma implies that if a set is ``essentially closed under finite or countably infinite addition'', then it is not a union of finitely many or countably many measured sets.
\end{remark}

The proof is based on the convolution of the measures ${\mu_j}|_{K_j}$ and Fubini's theorem. 

\begin{proof}
Let $\mu$ be the product of the measures ${\mu_j}|_{K_j}$ on the compact product space $\prod_{j\in J} K_j$. Let
\begin{align*}
\pi&:\prod_{j\in J} K_j \to \left[0, \,\sum_j a_j\right]\\
\pi&: (x_j)_{j\in J} \mapsto \sum_{j\in J} x_j.
\end{align*}
This is a continuous map. The image of $\mu$ under $\pi$, $\pi_*(\mu)$, is the convolution of the measures ${\mu_j}|_{K_j}$. Clearly, $\pi_*(\mu)$ is a probability measure supported by the compact set $K$.

Let $T\subset \R$ be uncountable, and assume that the sets $K+t$ ($t\in T$) are disjoint. Assume that $B_j$ are Borel sets with $K+T\subset \cup_{j\in J} B_j$.

For every $t\in T$, $K\subset \cup_j (B_j-t)$. Therefore there is a $j(t)\in J$ such that
$$\pi_*(\mu)(B_{j(t)}-t)>0,$$
that is,
$$\mu(\pi^{-1}(B_{j(t)}-t))>0.$$

Since $J$ is countable and $T$ is uncountable, there is $k\in J$ and an uncountable $T'\subset T$ such that
\begin{equation}\label{pos}
\mu(\pi^{-1}(B_{k}-t))>0
\end{equation}
for every $t\in T'$.

By Fubini's theorem, if $\mu(A)>0$ for a Borel set $A\subset \prod_j K_j$, then there are $x_j\in K_j$ ($j\in J\setminus\{k\}$) such that
$$\mu_k(\{x_k\in K_k \,:\, (x_j)\in A\})>0.$$ 
Using this with \eqref{pos},
$$\mu_k(\{x_k\in K_k \,:\, x_k + u \in B_k-t\})>0$$
and 
$$\mu_k(K_k \cap (B_k-t-u)\})>0$$
where $u\in\R$ is obtained in a form $u=\sum_{j\in J, \ j\neq k} x_j$ with $x_j\in K_j$.
Using that $\mu_k$ is translation invariant and $u\in \sum_{j\in J\setminus\{k\}} K_j$, 
\begin{equation}\label{pos2}
0<\mu_k((K_k+u+ t) \cap B_k) \le \mu_k((K+t)\cap B_k).
\end{equation}
Since the sets $K+t$ ($t\in T'$) are disjoint and $T'$ is uncountable, from \eqref{pos2} we conclude that $B_k$ is not $\sigma$-finite with respect to $\mu_k$.
\end{proof}

\begin{theorem}\label{constr}
Let $2\le n_1 \le n_2 \cdots$ be a sequence of integers tending to infinity. Let 
$$A=\left\{\sum_{i=1}^\infty \frac{k_i}{n_1 \cdots n_i}\,:\,
k_i\in\{0,1,\ldots, n_i-1\} \text{ and }
\sum_{i=1}^\infty \frac{k_i}{n_i} \le 1/2
\right\}.$$
Then $A$ is a non-empty compact set which is not a union of countably many measured sets.
\end{theorem}

The proof is based on Lemma~\ref{fubini-mix} and the following property of $A$. Whenever sets $A_j\subset A$ are given,
there is a translate $B_j$ of a `large part' of $A_j$ such that $\sum B_j \subset A$; moreover, there is a perfect compact set $P$ such that $P+\sum B_j\subset A$ where the translates of $\sum B_j$ are pairwise disjoint.

\begin{proof}
The set $A\subset [0,1]$ is clearly compact and has the cardinality of the continuum.

Assume that $A=\bigcup_{j=1}^\infty A_j$ where each $A_j$ is measured by a translation invariant Borel measure $\mu_j$.
%
As every finite and $\sigma$-finite Borel measure is inner regular, there are compact sets $K_j\subset A_j$ such that $0<\mu_j(K_j)<\infty$.

Every $x\in A$ can be uniquely expressed in the form
$$x=\sum_{i=1}^\infty \frac{k_i(x)}{n_1 \cdots n_i} \qquad
(k_i\in\{0,1,\ldots, n_i-1\}),$$
for which we also have $\sum k_i(x)/n_i \le 1/2$.

Our first aim is to replace each $K_j$ by a large compact subset $K'_j$ on which $\sum k_i(x)/n_i$ is uniformly convergent.
Let $t_m(x)$ denote the smallest positive integer $r$ such that
\begin{equation}\label{e1}
\sum_{i=r}^\infty \frac{k_i(x)}{n_i} \le 4^{-m}.
\end{equation}
Then $t_m:A\to \{1,2, \ldots\}$ is not continuous, but $\{x\in A \,:\, t_m(x)\le r\}$ is compact for every integer $r$.

For each $j$, by an induction argument, we can choose $r^j_1, r^j_2, \ldots$ so large that
the set
\begin{equation}\label{k'j}
K'_j=\{x\in K_j \,:\, \forall m\ge 1 \ t_{m}(x)\le r^j_m\}
\end{equation}
satisfies $\mu_j(K'_j)>0$. (We remark that this uniformity assumption on points of $K'_j$ is the same as requiring that the image of $K'_j$ under the map $x\mapsto (x, \sum k_i(x)/n_i) \in \R^2$ is compact.)


Now we would like to replace $K'_j$ by a large compact set $K''_j$ for which \eqref{e2} holds.

Define $h_r: A\to A$ as the map
$$\sum_{i=1}^\infty \frac{k_i}{n_1\cdots n_i} \mapsto \sum_{i=r}^\infty \frac{k_i}{n_1\cdots n_i}.$$
Let $$K_j''=h_{r^j_j}(K_j').$$ 
That is, we divide $K_j'$ into finitely many compact sets, translate them, and their union is $K_j''$. Therefore, by the translation invariance of $\mu_j$, we have $\mu_j(K''_j)>0$.
By \eqref{e1} and the choice of $r^j_j$, we also have 
\begin{equation}\label{e2}
\sum_{i=1}^\infty k_i(x)/n_i \le 4^{-j} \text{ \ for every }x\in K''_j.
\end{equation}
Since $\sum_j 4^{-j} = 1/3 < 1/2$, we have
$$\sum_{j=1}^\infty K_j'' \subset A$$
and 
\begin{equation}\label{est4}
\sum_{i=1}^\infty k_i(x)/n_i \le \sum_{j=1}^\infty 4^{-j} = 1/3 \text{ \ for every }x\in \sum_{j=1}^\infty K_j''.
\end{equation}
Note also that \eqref{k'j} implies that
\begin{equation}\label{est2}
\sum_{i=r^j_m}^\infty k_i(x)/n_i \le 4^{-m} \text{ \ for every }x\in K''_j \text{ and } m\ge 1.
\end{equation}

Let $$K=\sum_{j=1}^\infty K''_j.$$ We will now find a perfect compact set $P$ such that $P+K\subset A$ where all the translated copies are disjoint.

Fix positive integers $m(1)<m(2)<m(3)<\cdots$ and positive integers $a(l)$ such that
$$m(l)\ge \max(r^1_{4l},r^2_{4l},\ldots,r^{l+2}_{4l})$$
and that
$$4^{-l}/5 \le \frac{a(l)}{n_{m(l)}} \le 4^{-l}/4$$
for $l=1,2, \ldots$.

For $j=1,2,\ldots, l+2$, inequality \eqref{est2} implies that
$$\sum_{i=m(l)}^\infty k_i(x)/n_i \le  4^{-4l} \text{ \ for every }x\in K''_j.$$
For $j\ge l+3$, we can use \eqref{e2} to conclude that for every $x\in K$,
\begin{align}\label{est3}
\sum_{i=m(l)}^\infty k_i(x)/n_i  & \le  (l+2) 4^{-4l} + \sum_{j=l+3}^\infty 4^{-j} \nonumber
\\ & \le 4^{-l} (l+2) 4^{-3l} + 4^{-l}/48  
\\ & \le 4^{-l}/12. \nonumber
\end{align}

Let $$P=\left\{\sum_{l=1}^\infty \frac{b(l)}{n_{m(l)}} \, : \, b(l)\in\{0, a(l)\} \text{ for every } l=1,2,\ldots\right\}.$$
We will show that $P+K\subset A$ and that $(p+K)\cap (q+K) =\emptyset$ whenever $p,q$ are distinct elements from $P$.

Since $$\sum_{l=1}^\infty \frac{a(l)}{n_{m(l)}} \le \sum_{l=1}^\infty 4^{-l}/4 = 1/12,$$
inequality \eqref{est4} implies that $P+K\subset A$, in particular,
\begin{equation*}\label{est5}
\sum_{i=1}^\infty k_i(x)/n_i \le 1/3 +1/12 < 1/2 \text{ \ for every }x\in P+K.
\end{equation*}

Let $p, q\in P$ with $p\neq q$. Let $b_p(l)$ and $b_q(l)$ be the integers corresponding to $p$ and $q$ in the definition of $P$. Let
$$l_0=\min \{l \,:\, b_p(l)\neq b_q(l)\}$$
and assume, without loss of generality, that $b_p(l_0)=0$ and $b_q(l_0)=a(l_0)$.
Then, for every $x\in p+K$, \eqref{est3} implies that
\begin{align*}
\sum_{i=m(l_0)}^\infty k_i(x)/n_i & \le  4^{-l_0}/12 + \sum_{l=l_0+1}^\infty a(l)/n_{m(l)}
\\ & \le 4^{-l_0}/12 + \sum_{l=l_0+1}^\infty 4^{-l-1} 
\\ & \le 4^{-l_0}/12 + 4^{-l_0}/12 < 4^{-l_0}/5.
\end{align*}
On the other hand, for every $x\in q+K$, we clearly have
$$\sum_{i=m(l_0)}^\infty k_i(x)/n_i  \ge a(l_0)/n_{m(l_0)} \ge 4^{-l_0}/5.$$
Therefore $p+K$ and $q+K$ are disjoint sets.

We proved that $A$ contains uncountably many disjoint translates of $K$. Recall that $0<\mu_j(K''_j)<\infty$, and we can renormalise these measures to have $\mu_j(K''_j)=1$. Therefore we can apply Lemma~\ref{fubini-mix}. We obtain that $A$ cannot be written in the form of $\cup_j A_j$, where each $A_j$ has $\sigma$-finite $\mu_j$ measure. This contradicts our initial assumption. So $A$ is not a union of countably many measured sets.
\end{proof}

\bigskip
\section{Measured sets in Polish groups and Banach spaces}\label{s3}

We will consider Borel subgroups $G$ of Polish groups.  We say that a (Borel) measure $\mu$ is translation invariant on $G$ if it is both left and right invariant.
(For non-Abelian groups, assuming only left or right invariance is not enough for the analogue of Lemma~\ref{fubini-mix} to hold.) The notion of measured sets is the same as before: $A\subset G$ is called measured if there is a translation invariant Borel measure on $G$ such that $\mu$ assigns positive and $\sigma$-finite measure to $A$.

Theorem~2.8 in \cite{marcitamas} proves that every (additive) Borel subgroup of $\R$ which is not $F_\sigma$ is not measured. The same proof also gives that if $G$ is a Borel subgroup of a Polish group, and $G$ is not $\sigma$-compact, then $G$ is not measured. Here we prove the following stronger statement.

\begin{theorem}\label{thmpolish} \ 
\begin{enumerate}
\item\label{i2}
Let $G$ be a Polish group. If $G$ is not locally compact, then $G$ is not a union of countably many measured sets.

\item\label{i1}
Let $G$ be a Borel subgroup of a Polish group. If $G$ is not $\sigma$-compact, then $G$ is not a union of finitely many measured sets.

\end{enumerate}
\end{theorem}

We remark that a Polish group is $\sigma$-compact if and only if it is locally compact. (Any separable and locally compact metric space is $\sigma$-compact. For the other direction, Baire category theorem implies that in every countable covering of $G$ with compact sets there is one with non-empty interior.)

To prove the theorem, we will use an analogue of Lemma~\ref{fubini-mix} in this group setting. To study infinite products, we need the completeness of the space, so the lemma is as follows.

\begin{lemma}\label{fubini-mix2}
Let $J$ be a finite set and $G$ be a Borel subgroup of a Polish group; or let $J$ be countable and $G$ be a Polish group.

Let $\mu_j$ ($j\in J$) be translation invariant Borel measures. Let $K_j\subset G$ be compact sets with $0<\mu_j(K_j)<\infty$.
If $J$ is infinite, assume that the infinite product
$$K=\prod_{j\in J} K_j \subset G$$ 
(in some order) exists, giving a compact set $K\subset G$.
Let $B$ be a Borel set containing uncountably many disjoint translates of $K$. Then there are no Borel sets $B_j$ with $B=\cup_j B_j$ where  every $B_j$ has $\sigma$-finite $\mu_j$-measure.
\end{lemma}
\begin{proof}
The proof is essentially the same as the proof of Lemma~\ref{fubini-mix}. It is easy to check that the earlier proof works even if the group is not Abelian.
\end{proof}

The following lemma and its proof are essentially contained by \cite[Theorem 2.8]{marcitamas} (there $G$ is a subgroup of $\R$).
\begin{lemma}\label{unctrans}
Assume that $G$ is a topological group which is not $\sigma$-compact. Then every compact set $K\subset G$ has uncountably many disjoint translates in $G$.
\end{lemma}
\begin{proof}
We define a transfinite sequence of points $\{t_\alpha \in G : \alpha < \omega_1\}$ by transfinite induction so that the sets $\{t_\alpha K: \alpha < \omega_1\}$ are pairwise disjoint. Clearly,
$$t_\alpha K \cap t_\beta K =\emptyset \text{ \ if and only if \ }t_\alpha \notin t_\beta K K^{-1}.$$ Therefore, at step $\alpha$ our task is to find $t_\alpha\in G$ such that $t_\alpha\notin t_\beta K K^{-1}$ for any $\beta<\alpha$. Since $K K^{-1}$ is compact, the set $\cup_{\beta<\alpha} t_\beta K K^{-1}$ is $\sigma$-compact and cannot cover $G$. So the induction works. 
\end{proof}

\begin{proof}[Proof of Theorem~\ref{thmpolish}]
First let us note that in every Polish space, every finite Borel measure is inner regular (that is, the measure of every Borel set can be approximated using compact subsets), see \cite[Theorem 7.1.7]{bogachev}. The same holds for measures defined on any Borel subset of any Polish space. 

To prove part \eqref{i1},
let $G$ be a not $\sigma$-compact Borel subgroup of a Polish group.
Assume that $G=\cup_{j=1}^n A_j$, where the sets $A_j$ are Borel and there are translation invariant Borel measures $\mu_j$ such that $0<\mu_j(A_j)$ and $\mu_j$ is $\sigma$-finite on $A_j$. Choose compact sets $K_j\subset A_j$ such that $0<\mu_j(K_j)<\infty$. Let $K=\prod_{j=1}^n K_j$. This is a compact set in $G$. Combining Lemma~\ref{unctrans} with Lemma~\ref{fubini-mix2} for $B=G$ gives that our initial assumption is false. This proves that $G$ is not a union of finitely many measured sets.

To prove part \eqref{i2}, let $G$ be a not $\sigma$-compact (not locally compact) Polish group. Assume that $G=\cup_{j=1}^\infty A_j$, where the sets $A_j$ are Borel and there are translation invariant Borel measures $\mu_j$ such that $0<\mu_j(A_j)$ and $\mu_j$ is $\sigma$-finite on $A_j$. Choose compact sets $K_j\subset A_j$ such that $0<\mu_j(K_j)<\infty$. It is easy to see that for each $j$, there is $x_j\in K_j$ such that the intersection of $K_j$ with any neighbourhood of $x_j$ has positive $\mu_j$ measure. We define $K'_j$ by translating $K_j$ by $x_j^{-1}$ and intersecting it with a small neighbourhood of the identity,
$$K'_j=K_j x_j^{-1} \cap B(1, \eps_j).$$
(Here the distance is the metric realising that $G$ is a Polish group.)
Then $\mu_j(K'_j)>0$ and the infinite group multiplication
$$K=\prod_{j=1}^\infty K'_j$$
makes sense and defines a compact set $K\subset G$ provided that $\eps_j\to 0$ sufficiently fast. (Note that $\eps_j$ might depend on $K'_1 K'_2 \cdots K'_{j-1}$. If the metric is translation invariant, we can take any $(\eps_j)$ of finite sum.)  We again use Lemma~\ref{unctrans} and Lemma~\ref{fubini-mix2} to obtain a contradiction.
\end{proof}

\begin{cor}\label{cor1}
Let $X$ be an infinite dimensional separable Banach space. Then $X$ is not a union of countably many measured sets. 
Moreover, every closed not locally compact subgroup of $X$ has the same property.  
\end{cor}

\begin{proof}
Every closed subgroup of $X$ is a Polish group (using the same metric). Theorem~\ref{thmpolish} implies the statement.
\end{proof}

Based on Corollary~\ref{cor1}, now we construct Borel and $\sigma$-compact additive subgroups of $\R$ which are not a union of countably measured sets.

We write $$X_r=\{x\in X : \norm{x}\le r\}$$
for the closed ball in the Banach space $X$ of radius $r$ centred at $0$.

\begin{theorem}\label{banachconstr}
Let $X$ be an infinite dimensional Banach space with Schauder basis $(e_i)$. Assume $\norm{e_i}=1$.
Let $2\le n_1 \le n_2 \le \cdots \to \infty$.
Define
\begin{align*} 
A=\Bigg\{ \sum_{i=1}^\infty \frac{k_i}{n_1 \cdots n_i} \in \R\,: & \ k_i\in\Z\text{ and }
 \sum_{i=1}^\infty \frac{k_i}{n_i}e_i \in X\Bigg\}
\end{align*}
where by $\sum (k_i/n_i)e_i\in X$ we mean that the sum converges in $X$; then the sum $\sum_i k_i/(n_1\cdots n_i)$ converges automatically.

Then $A$ is a Borel additive subgroup of $\R$ and is not a union of countably many measured sets.

If the Schauder basis is boundedly complete, then $A$ is $\sigma$-compact.

\end{theorem}

\begin{remark}
For every $1\le p<\infty$, the space $\ell_p$ (with the standard basis) is boundedly complete, therefore the obtained $A$ is $\sigma$-compact and not a union of countably many measured sets.
\end{remark}
\begin{remark}\label{remAr}
For every $r>0$, let
\begin{align*} 
A_r=\Bigg\{ \sum_{i=1}^\infty \frac{k_i}{n_1 \cdots n_i} \in \R\,: & \ k_i\in\Z\text{ and }
 \sum_{i=1}^\infty \frac{k_i}{n_i}e_i \in X_r\Bigg\}
\end{align*}
where by $\sum (k_i/n_i)e_i\in X_r$ we mean that the sum converges in $X$ to a point in $X_r$.
Then $A_r$ is not a union of countably many measured sets. Also, the closure of $A_r$ is in $A$. For $\ell_p$ spaces with the standard basis (and whenever the basis is monotone and boundedly complete), $A_r$ is compact. (See the proof of Theorem~\ref{banachconstr} for proofs.)
\end{remark}

\begin{remark}\label{rem1}
In the set
\begin{align*} 
E=\Bigg\{ \sum_{i=1}^\infty \frac{k_i}{n_1 \cdots n_i} \in \R\,: & \ k_i\in\{-m_i, \ldots, m_i\}\Bigg\}
\end{align*}
every point has a unique representation of the form $\sum_{i=1}^\infty k_i/(n_1 \cdots n_i)$ if $2m_i+1<n_i$. If $n_i=2m_i+1$, then $E=[-1/2, 1/2]$.
\end{remark}

\begin{proof}[Proof of Theorem~\ref{banachconstr}]
As $(e_i)$ is a Schauder basis, every $v\in X$ has a unique representation as
$$v=\sum_{i=1}^\infty x_i e_i \quad (x_i\in \R).$$
It is well known that there is a constant $C>0$ (depending on the basis only) for which $\abs{x_i}\le C\norm{v}$, for every $i$ and $v\in X$.

Therefore $\sum_i (k_i/n_i)e_i\in X_r$ implies that $\abs{k_i/n_i}\le Cr$, and thus $\sum_i k_i/(n_1\cdots n_i)$ indeed always converges.

Assume $r<1/(3C)$. Then $\sum (k_i/n_i)e_i \in X_r$ implies that $\abs{k_i/n_i} < 1/3$ and therefore $k_i\in\{-m_i, \ldots, m_i\}$ with $m_i=\lceil n_i/3 \rceil-1$, so $k_i$ can take only less than $n_i$ different values. Remark~\ref{rem1} implies that in this case every point in $A_r$ (defined in Remark~\ref{remAr}) has a unique representation of the form
$$x=\sum_{i=1}^\infty \frac{k_i}{n_1\cdots n_i} \quad (k_i\in\{-m_i, \ldots, m_i\}).$$

Let 
$$G=\left\{\sum_{i=1}^\infty \frac{k_i}{n_i}e_i \in X\,:\, k_i\in \Z\right\}$$
and for $r>0$,
$$G_r=\left\{\sum_{i=1}^\infty \frac{k_i}{n_i}e_i \in X_r\,:\, k_i\in \Z\right\}.$$
As $(e_i)$ is a basis, $G$ is weakly closed, and thus closed. In fact, $G$ is a closed subgroup of $X$; it is the closure of the subgroup generated by the vectors $e_i/n_i$. 
On the other hand, $G$ is not locally compact. Indeed, for any $h>0$ there are integers $k_i$ such that $\norm{(k_i/n_i)e_i}\to h$. Any subsequence of $(k_i/n_i)e_i$ either converges weakly to $0$ or it does not even converge weakly, but it does not converge in norm to $0$.

The set $G_r=G\cap X_r$ is also closed. 

Let $f_r:G_r\to A_r$ be the map for which
$$f_r\left( \sum_{i=1}^\infty\frac{k_i}{n_i}e_i
  \right) = \sum_{i=1}^\infty \frac{k_i}{n_1\cdots n_i}.$$
This is well defined, and previous arguments imply that this is a bijection if $r<1/(3C)$. As $(e_i)$ is a Schauder basis, $f_r$ is continuous. An injective continuous image (of a closed subspace) of a Polish space is Borel, hence $A_r$ is Borel and $f_r^{-1}$ is Borel when $r<1/(3C)$. Notice that as $G$ is separable, it is a union of countably many translates of $G_r$. Since $A$ is a linear image of $G$, it is an additive subgroup of $\R$, and $A$ is a union of countably many translates of $A_r$. Therefore $A$ is Borel.

\begin{claim}\label{claimcompact}
If $(e_i)$ is boundedly complete, then the closure of $A_r$ is in $A_{C'r}$, hence $A$ is $\sigma$-compact.

If $(e_i)$ is the standard basis of $\ell_p$ (in general, a monotone and boundedly complete basis), then $C'=1$ and $A_r$ is compact.
\end{claim}
\begin{proof}

Let $x_j\in A_r$ ($j=1,2,\ldots$). Then there are integers $k_i^j$ with 
$$\sum_{i=1}^\infty \frac{k_i^j}{n_i}e_i \in G_r$$ such that
$$f_r\left( \sum_{i=1}^\infty\frac{k_i^j}{n_i}e_i
  \right) = \sum_{i=1}^\infty \frac{k_i^j}{n_1\cdots n_i} = x_j.$$
As $(e_i)$ is a Schauder basis, there is $C'\in [1,\infty)$ such that
 $$\norm{\sum_{i=1}^m \frac{k_i^j}{n_i} e_i}\le C' \norm{ \sum_{i=1}^\infty \frac{k_i^j}{n_i} e_i} \le C'r.$$
We have $C'=1$ when the basis is monotone by definition.
 
We have $\lvert k_i^j/n_i \rvert \le Cr$. 
By passing to a subsequence of $(x_j)$ we may assume that $k^j_i$ converges for every $i$ as $j\to\infty$. Let
$k_i=\lim_{j\to\infty} k^j_i.$
Then
$$\norm{\sum_{i=1}^m \frac{k_i}{n_i} e_i}\le C'r \quad (m\ge 1).$$
Then being boundedly complete implies that the sum $\sum_{i=1}^\infty (k_i/n_i)e_i$ converges, obviously to a point $y\in X_{C'r}$. Clearly, 
$$f_{C'r}(y) = \sum_{i=1}^\infty \frac{k_i}{n_1\cdots n_i} \in A_{C'r}$$ is the limit of the subsequence of the original $(x_j)$.
\end{proof}

\begin{claim}\label{claimsmallr}
For $r<1/(3C)$, $A_r$ is not a union of countably many measured sets.
\end{claim}
\begin{proof}
Assume that $A_r$ is a union of countably many measured sets. Let $\mu_j$ be translation invariant Borel measures on $\R$ for which there are Borel sets $A_j$ with $A_r=\cup_j A_j$, $\mu_j(A_j)>0$ and $\mu_j$ is $\sigma$-finite on $A_j$.

Let $\nu_j$ be the image of the measure ${\mu_j}|_{A_r}$ under $f_r^{-1}$. Then $\nu_j$ is a Borel measure on $G_r$, it is $\sigma$-finite on $f_r^{-1}(A_j)$, we have $\nu_j(f_r^{-1}(A_j))>0$ and $\cup_j f_r^{-1}(A_j)=G_r$.

Notice that $f_r$ basically preserves the group structure in the sense that
$$\text{if \ }x,y,x+y\in G_r \text{ \ then \ }f_r(x)+f_r(y)=f_r(x+y).$$
Therefore,
if $S\subset G_r$, $v\in G$ and $S+v\subset G_r$, then
\begin{equation}\label{inv}
\nu_j(S+v) = \mu_j(f_r(S+v)) = \mu_j(f_r(S)+f_r(v)) = \mu_j(f_r(S)) = \nu_j(S).
\end{equation}
So $\nu_j$ is ``translation invariant inside $G_r$''. Therefore we can extend $\nu_j$ to a translation invariant measure on $G$ in the following way. For a Borel set $S\subset G$ define
$$\nu^*_j(S) = \sum_{k=1}^\infty \nu_j(S_k+v_k)$$
where $\{S_k\}$ is a Borel partition of $S$ and $v_k\in G$ such that $S_k+v_k \subset G_r$ for every $k\ge 1$. Such choice of $S_k$ and $v_k$ exists because $X$ is separable, countably many translates of $G_r$ covers $G$. Property \eqref{inv} of $\nu_j$ implies that $\nu_j^*(S)$ is well-defined. It is easy to check that $\nu^*_j$ is indeed a measure.

Fix $v_k\in G$ such that $\cup_{k=1}^\infty(G_r+v_k)= G$. Let $A^*_j=\cup_k (f_r^{-1}(A_j)+v_k)$. Then $G=\cup_j A^*_j$, $\nu^*_j(A^*_j)>0$ and $\nu^*_j$ is $\sigma$-finite on $A^*_j$. Thus $G$ is a union of countably many measured sets. But this contradicts Corollary~\ref{cor1} since $G$ is a closed not locally compact subgroup of $X$. Therefore $A_r$ is not a union of countably many measured sets if $r<1/(3C)$.
\end{proof}

\begin{claim}\label{claimallr}
For every $r>0$, $A_r$ and $A$ are not a union of countably many measured sets.
\end{claim}
\begin{proof}
Let $0<s<1/(3C)$ and $r\in (s,\infty]$, where $A_\infty=A$ and $G_\infty=G$.
Since $G_r$ is covered by countably many translates of $G_s$, we see that $A_r$ is also covered by countably many translates of $A_s$ and $A_s\subset A_r$. It is easy to check that if $A_r$ was a union of countably many measured sets, then so would be $A_s$. Claim~\ref{claimsmallr} implies our claim.
\end{proof}
Claim~\ref{claimcompact} and Claim~\ref{claimallr} conclude the proof of Theorem~\ref{banachconstr}.
\end{proof}

\bigskip

\section{Typical sets and measured sets}\label{s4}

In this section we prove that there is a measured (compact) set in $\R$ which is null or non-$\sigma$-finite for every Hausdorff measure (with arbitrary gauge function).
We also show that `many' $C^1$ images of sets similar to those constructed in Section~\ref{s2} and \ref{s3} are measured, but not a union of countably many sets which are measured by Hausdorff measures.

We start with a few lemmas.

\begin{lemma}
Let $A\subset \R$ and $f:A\to \R$ be Lipschitz with Lipschitz constant $L$. Then $$\iH^g(f(A))\le \lceil L\rceil\cdot\iH^g(A)$$ for any gauge function $g$.
\end{lemma}
\begin{proof}
Consider an arbitrary covering $\{U_i\}$ of $A$. For each $i$, 
$$\diam\,f(U_i) \le L \cdot \diam\,U_i.$$
Cover $f(U_i)$ with $\lceil L \rceil$ many intervals of length $\diam\,U_i$. The union of all these intervals cover $f(A)$. Therefore $\iH^g(f(A))\le \lceil L\rceil\cdot\iH^g(A)$.
\end{proof}

\begin{cor}\label{hauseq}
If $A\subset \R$ is measured by a Hausdorff measure (with respect to some gauge function), then any bi-Lipschitz image of $A$ is measured as well (using the same Hausdorff measure).\qed
\end{cor}

The following is an important observation of M.~Elekes and T.~Keleti, see \cite[Lemma 2.17]{kiterjesztes} for a general result.

\begin{lemma}\label{xyuv}
Let $\emptyset \neq A\subset \R$ be Borel such that $A\cap (A+t)$ consists of at most $1$ point for every $t\neq 0$. (In other words, the equation $x-y=u-v$ only has trivial solutions in $A$.) Then $A$ is measured.
\end{lemma}
\begin{proof}
If $A$ is countable, then the counting measure measures $A$. Otherwise $A$ contains non-empty perfect sets and thus supports non-atomic Borel probability measures. Any such measure $\mu$ on $A$ can be extended to be a translation invariant Borel measure $\mu^*$ on $\R$ for which $\mu^*(A)=1$, see \cite[Lemma 2.17]{kiterjesztes}.
\end{proof}

\begin{theorem}
There is a compact set in $\R$ which is measured by some translation invariant Borel measure, but it is not measured by any Hausdorff measure.
\end{theorem}
\begin{proof}
It follows from a result of J.~von~Neumann \cite{neumann} that there is an algebraically independent non-empty perfect set in $(2, \infty)$.

We can find non-empty disjoint perfect sets $P_i$ ($i=1,2,\ldots$) in $P$.
Let 
$$A=\bigcup_{n=1}^\infty P_1 P_2 \cdots P_n$$
where products of sets is defined as $S T = \{st:s\in S, t\in T\}$. Then $A\subset (2,\infty)$ is clearly closed. Since $P$ is algebraically independent, the equation $x-y=u-v$ in $A$ only has trivial solutions. Therefore $A$ is measured by Lemma~\ref{xyuv}.

On the other hand, let
$$B=\log A = \bigcup_{n=1}^\infty (\log P_1 + \log P_2 +\ldots + \log P_n)$$
where $\log S = \{\log s : s\in S\}$. We claim that $B$ is not measured. (This will be very similar to the argument Davies used in \cite{davies}.)
Indeed, let $\mu(B)>0$ for a translation invariant Borel measure. Then there is $n$ such that
$\mu(\log P_1 + \ldots +\log P_n)>0$. Since $P$ is algebraically independent, the sets
$$\log P_1 + \ldots +\log P_n + t \quad (t\in \log P_{n+1})$$
are all disjoint. Since $\log P_{n+1}$ is uncountable,
$\mu$ must be non-$\sigma$-finite on $\log P_1 + \ldots +\log P_{n+1}$, hence also on $B$. 

Since $A$ is a bi-Lipschitz image of $B$, the set $A$ cannot be measured by any Hausdorff measure by Lemma~\ref{hauseq}.

The set $A$ is closed but not compact. To obtain a compact example, take
\[
\bigcup_{n=1}^\infty \Big(\big(A\cap [\sqrt{n}, \sqrt{n+1})\big) -\sqrt{n}\Big).\qedhere
\]
\end{proof}

Now we aim at the stronger theorem that there are compact sets which are measured, but not a union of countably many sets which are measured by Hausdorff measures.
 

\begin{lemma}\label{typimage}
Let $E\subset [0,1]$ be a compact set of lower box dimension less than $1/4$. Then for a typical $C^1$ function $f:[0,1]\to \R$, $f(E)$ is measured.
\end{lemma}
\begin{proof}[Sketch of proof.]
A result \cite{where} of Z.~Buczolich and the present author states that for every compact set $E\subset [0,1]$ with lower box dimension less than $1/2$, a typical $C^1$ function $f:[0,1]\to \R$ is injective on $E$. In other words, if $f(x)-f(y)=0$ with $x,y\in E$, then $x=y$. One can similarly show that if the lower box dimension of $E$ is less than $1/4$, then for a typical $f\in C^1([0,1])$, the equation $$f(x)-f(y)=f(u)-f(v) \qquad (x,y,u,v\in E)$$ only has trivial solutions, that is, $x=u$ and $y=v$, or $x=y$ and $u=v$. 
In particular, the equation $x-y=u-v$ in $f(E)$ only has trivial solutions. Then Lemma~\ref{xyuv} implies that $f(E)$ is measured.
\end{proof}

\begin{lemma}\label{daviesbox}
There is a non-empty compact set $E\subset \R$ of box dimension $0$ which is not a union of countably many measured sets.
\end{lemma}
\begin{proof}
The set $A$ in Theorem~\ref{constr} can be easily modified to have box dimension zero.
Let $2\le n_1 \le n_2 \cdots \to\infty$, and let $N_i\ge n_i^i$. Then the same proof shows that 
$$A'=\left\{\sum_{i=1}^\infty \frac{k_i}{N_1 \cdots N_i}\,:\,
k_i\in\{0,1,\ldots, n_i-1\} \text{ and }
\sum_{i=1}^\infty \frac{k_i}{n_i} \le 1/2
\right\}$$
is not a union of countably many measured sets. It is easy to check that $A'$ has box dimension zero.
\end{proof}

\begin{theorem}\label{measnohaus}
There is a compact set in $\R$ which is measured by some translation invariant Borel measure, but it is not a union of countably many sets which are measured by some Hausdorff measures.
\end{theorem}
\begin{proof}
Let $E$ be the compact set given by Lemma~\ref{daviesbox}. Assume $E\subset [0,1]$. Lemma~\ref{typimage} implies that for a typical $f\in C^1([0,1])$, $f(E)$ is measured.   Therefore there is also such an $f$ close in $C^1$ norm to the identity $x\mapsto x$. In particular, there is a bi-Lipschitz $C^1$ function $f$ for which $f(E)$ is measured.

Since $E$ is not a union of countably many measured sets and $f$ is bi-Lipschitz, Corollary~\ref{hauseq} implies that $f(E)$ is not a union of countably many sets which are measured by some Hausdorff measures.
\end{proof}


In a sense, we proved that in a carefully chosen category of compact sets a typical set satisfies Theorem~\ref{measnohaus}.
We finish by noting that if we consider non-empty compact sets in the Hausdorff metric, then typical compact sets are measured: they even satisfy Lemma~\ref{xyuv}. Moreover, typical compact sets are measured by a Hausdorff measure as well; this is a theorem of R.~Balka and the present author \cite{tipikus}.

\begin{theorem}
For a typical compact set $K\subset \R$ in the sense of Baire category (in the complete metric space of non-empty compact sets with the Hausdorff distance) there is a gauge function $g$ with $\iH^g(K)=1$.\qed
\end{theorem}

Note also that for every fixed gauge function $g$ with $\lim_{x\to 0} g(x)=0$, the typical compact set $K$ satisfies $\iH^g(K)=0$.

\bigskip

\noindent{\textbf{Acknowledgement.}} The author is grateful to M\'arton Elekes and Tam\'as Keleti for many helpful discussions and for bringing the open problems to his attention.

\end{document}